\documentclass[11pt]{amsart}
\usepackage{amssymb}
\usepackage{epsfig}
\usepackage{graphicx}

\input xy
\xyoption{all}

\theoremstyle{plain}
\newtheorem{thm}{Theorem}[section]

\newtheorem{coro}[thm]{Corollary}

\newtheorem{prop}[thm]{Proposition}

\theoremstyle{definition} \theoremstyle{definition}
\newtheorem{defn}[thm]{Definition}

\newtheorem{exam}[thm]{Example}

\newtheorem{rem}[thm]{Remark}           

\theoremstyle{remark}

\def\vol{{\rm vol}}

\newcommand{\De}{\Delta}
\newcommand{\de}{\delta}
\newcommand{\La}{\Lambda}

\newcommand{\al}{\alpha}
\newcommand{\C}{\mathbb C}
\newcommand{\Q}{\mathbb Q}

\newcommand{\R}{\mathbb R}

\def\cA{{\mathcal A}}
\def\cB{{\mathcal B}}

\def\cD{{D}}

\def\cT{{\mathcal T}}

\def\bdelta{{\boldsymbol{\delta}}}
\def\bDelta{{\boldsymbol{\Delta}}}
\def\bLambda{{\boldsymbol{\Lambda}}}

\def\sT{{\mathsf T}}
\def\sL{{\mathsf L}}

\def\sK{{\mathsf K}}

\def\G{{\mathsf G}}

\def\M{{\mathsf M}}

\def\P{{\mathsf P}}
\def\T{{\mathsf T}}

\def\B{{\mathsf B}}

\newcommand{\Ker}{{\rm Ker}}

\newcommand{\K}{{\mathsf K}}

\newcommand{\PGL}{{\rm PGL}}
\def\Val{{\rm Val}}

\def\ra{\rightarrow}

\def\A{{\mathbb A}}
\def\C{{\mathbb C}}

\def\bP{{\mathbb P}}
\def\Q{{\mathbb Q}}

\def\C{{\mathbb C}}
\def\N{{\mathbb N}}

\def\mX{{\mathfrak X}}

\def\Pic{{\rm Pic}}

\def\PGL{{\rm PGL}}

\def\zZ{{\mathcal Z}}
\def\eps{{\epsilon}}

\def\bfs{{\mathbf s}}

\def\sA{{\mathsf A}}

\def\sT{{\mathsf T}}
\def\sL{{\mathsf L}}

\def\sK{{\mathsf K}}

\begin{document}

\author{Ramin Takloo-Bighash}
\address{Department of Mathematics, UIC\\
851 S. Morgan str.\\
Chicago, IL 60607, U.S.A.}
\email{rtakloo@math.uic.edu}

\author{Yuri Tschinkel}
\address{Courant Institute of Mathematical Sciences, N.Y.U. \\
 251 Mercer str. \\
 New York, NY 10012, U.S.A.}
\email{tschinkel@cims.nyu.edu}

\title[Integral points]{Integral points of bounded height on compactifications of semi-simple groups}

\

\

\begin{abstract}
We study the asymptotic distribution of integral points of bounded height
on partial bi-equivariant compactifications of semi-simple groups of adjoint type.
\end{abstract}

\date{\today}

\maketitle

\vspace*{.5pc}
\hfill
{\it To Joe Shalika}

\tableofcontents

\section{Introduction}
\label{sect:intro}

In this paper, we study the asymptotic distribution of $S$-integral points of bounded height
on partial bi-equivariant compactifications of semi-simple groups of adjoint type.
Our approach combines the spectral techniques developed in \cite{STT} 
in the context of integral points with more recent results in 
\cite{C-T-integral}, \cite{C-T-Ga}, and \cite{C-T-Torus}.

Throughout, $\G$ is a split semi-simple group of adjoint type of rank larger
than one over a number field $F$ and $X$ the wonderful compactification
of $\G$. 
The boundary $X \setminus \G$ is a strict normal crossings divisor 
$\cup_{\alpha\in \mathcal A}D_{\alpha}$ whose set of irreducible components $\cA$
is in bijection with the set $\Delta=\Delta(\G,\T)$ of simple roots of $\G$, with respect to 
a fixed maximal split torus $\T$ of $\G$. 
Each $A\subseteq \Delta$ corresponds to the boundary stratum
$D_A:=\cap_{\alpha\in A} D_{\alpha}$. 
A rational character $\lambda \in \mathfrak X^*(\T)$ gives rise to a line bundle
$L_{\lambda}$ on $X$, the classes of such line bundles span the Picard group 
$\Pic(X)$ of $X$, after tensoring with $\Q$ (see \cite[Section 5]{STT}).
In \cite{STT} we defined a height pairing 
$$
H\colon \mathfrak X^*(\T)_{\C} \times \G(\A)  \rightarrow \C
$$
whose restriction to $\lambda \times \G(F)$ is a 
a standard Weil height function on $X$ with respect to $L_{\lambda}$. 
The height function is left- and right invariant with respect to the 
action of a maximal compact subgroup $\K=\prod_v \K_v\subset \G(\A)$ of the adelic points of $\G$.
We established an asymptotic formula for the number
$$
N(B,\lambda):=\# \{ \gamma\in \G(F) \,\mid \, H(\lambda,\gamma)\le B \},  \quad B\ra \infty,
$$
for all $\lambda$ corresponding to big line bundles on $X$, i.e., 
all classes in the interior of the effective cone in $\Pic(X)_{\R}$. 

From now on we fix a divisor 
$D\subset X\setminus \G$ and let $\mathcal A_D$ be the set of 
$\alpha$ such that 
$D=\cup_{\alpha\in \mathcal A_D} D_{\alpha}$. 
Let $S$ be a finite set 
of places of $F$ containing the archimedean places. We introduce the set-theoretic
characteristic function $\bdelta=\bdelta_{D,S}$ 
of the set of $(D,S)$-integral points on $X$ as follows:
For $v\notin S$, let
$$
\bdelta_v \colon \G(F_v)\ra \{ 0,1\}
$$
be given by 
$$
\bdelta_{v}(g_v) = 
\begin{cases} 
1 & \alpha(t_v) =1,
\forall \alpha \in \mathcal A_{\cD};\\ 
0 & \text{ otherwise},
\end{cases}
$$
where 
$$
g_v = k_v t_v k'_v, \quad \text{  with  } \quad t_v \in \T(F_v), k_v, k'_v \in \K_v.
$$ 
For $v\in S$, we put $\bdelta_v\equiv 1$ and write
$$
\bdelta:=\prod_v \bdelta_v.
$$ 
A point $\gamma \in \G(F)\subset \G(\A)$ is called $(\cD,S)$-integral if
$\bdelta(\gamma)$ is equal to $1$. This definition agrees with the usual definition 
of $S$-integral points on a model of $X\setminus D$ over the integers.
The corresponding counting functions are given by
$$
N_{S, \cD}(B,\lambda) =\# \{   (\cD,S)-\text{integral}\quad 
\gamma \in \G(F) \,\mid \, H(\lambda,\gamma)\le B\}.
$$
Our main result is Theorem~\ref{thm:main2} which shows that
$$
N_{S, \cD}(\lambda,B) = c\cdot B^a \log(B)^{b-1}(1+o(1)),   \quad B\ra \infty,
$$ 
with an explicit bound on the error term. 
The constants $a, b,$ and $c$ involve arithmetic, geometric, and combinatorial
information about the data $X,D$, $\lambda$, and $S$. 
Our Theorem~\ref{thm:main2} can be considered as an interpolation between results 
on {\em rational points} in \cite{STT} and \cite{GMO} and on {\em integral} 
points on $\G$ in \cite{GOS}, \cite{gorod-nevo}.  
It is consistent with a generalization of the conjecture of Batyrev-Manin and Peyre
to the framework of integral points proposed in 
\cite{C-T-integral}, \cite{C-T-Ga}, and \cite{C-T-Torus}.

\

As in \cite{STT} our method is based on the spectral analysis of
the height zeta function, modified to count $(\cD,S)$-integral points. 
We prove that the main poles of the height zeta function arise from 
the contribution of unramified one-dimensional automorphic 
representations of $\G(\A)$,
the set of which is denoted by $\mathcal X(\G)$.

\

In rank one, i.e., for $\G =\PGL(2)$, the 
spectral theory and the height integrals have
been completely worked out in \cite{PGL2}.
The wonderful compactification of $\PGL(2)$ is
the projective space $\bP^3$, with a unique irreducible 
boundary component which is a smooth split quadric, 
and the distribution of $(\cD,S)$-integral points 
is an interesting classical problem.

\

The study of $S$-integral points on algebraic varieties has a rich
history, going back at least as far as Siegel's work on algebraic
curves. More recently the distribution of integral points on
affine homogeneous varieties has been studied in several papers, e.g., in
\cite{duke-rudnick-sarnak, H-L-1,H-L-2}, 
as well as in \cite{E-M} where ergodic-theoretic methods are
employed. Ergodic-theoretic methods are also used in \cite{B-H} to
treat fairly general homogeneous varieties. Our method, in
technique and spirit, is quite close to \cite{C-T-Ga, C-T-Torus}.
An introduction to integral points for pairs $(X, D)$
can be found \cite{Hassett-Tschinkel-integral}. 
For background concerning the geometry of wonderful compactifications, see
\cite{D-P}. Arithmetic properties of these varieties,
in connection with counting of rational points of bounded height,
are discussed in
\cite{STT}. 

\

The paper is organized as follows. 
In Section~\ref{sect:heights} we review the theory of height functions from 
the group-theoretic perspective and study local and global 
height integrals. Section~\ref{sect:regularintegral} is devoted to the regularization of
adelic height integrals. Section \ref{sect:heightzeta} contains
our main theorem and its proof.

\

{\bf Acknowledgments.}
We are grateful to Wee-Teck Gan for helpful remarks. 
The first author was partially supported by NSF grant DMS-0701753.
The second author was partially supported by NSF grants DMS-0739380 and 
0901777.

\section{Basic notation}
\label{sect:notions}

For $X$ an algebraic variety over a field $F$ we write $X(F)$ for
the set of $F$-rational points of $X$. 
We denote by $\Pic(X)$ the Picard group of $X$ and 
by $\La_{\rm eff}(X)\subset \Pic(X)_{\R}$ the (closed) cone of effective
divisors on $X$. We often identify line bundles, divisors and
their classes in $\Pic(X)$.

From now on, we let $F$ be a number field, $\Val(F)$ the set its places, and
$S_{\infty}$ the set of archimedean places. For
$v\in \Val(F)$, let $F_v$ be the completion of $F$ with respect to
$v$, $\varpi_v$ a uniformizer at $v$,  and $q=q_v$ the order of the residue field. 
For any finite set of places $S$ (containing
$S_{\infty}$) we denote by $\mathfrak o_S$ the ring of $S$-integers of
$F$. We denote by $\A$ the ring of adeles, by 
$\A_S=\prod_{v\notin S}' F_v$ and by $\A_f=\A_{S_{\infty}}$. 
All Haar measures $\mathrm dg_v$ will be normalized as in \cite{Arthur}; 
we put $\mathrm dg_S= \prod_{v\notin S} \mathrm dg_v$ and 
$\mathrm dg = \prod_v \mathrm dg_v$.

\

We fix a connected split semi-simple
group $\G$ of adjoint type over $F$ or rank $r\ge 2$.  Let $\sT$
be a maximal split torus of $\G$ and  $\B\subset \G$
a Borel subgroup containing $\sT$. 
We also let $\Delta=\Delta(\G,T)$ 
be the set of simple roots for $(\G, \T)$, with respect to $\B$, 
and $\rho$ half the sum of positive roots. We have
$$
2\rho = \sum_{\alpha\in \De} \kappa_{\alpha} \alpha, 
$$
for some constants $\kappa_{\alpha}\ge 1$.
We set
$$
\hat{F}_v := \begin{cases}  \quad
 \{ x \in \R \, |\,  x \geq 1 \}  & \text{ if } v\mid \infty ,\\
  \quad  \{ \varpi_v^{-n} \, | \, n \in \N \} &  \text{ if } v\nmid \infty,
\end{cases}
$$
and
\begin{equation}
\T(F_v)^+ = \{ a \in \T(F_v) \, | \, \alpha(a) \in \hat{F}_v
\text{ for each } \alpha \in \Phi^+ \}.
\end{equation}
We have the following Cartan decomposition:
\begin{equation}
\G(F_v) = \K_v \T(F_v)^+ \K_v.
\end{equation}

\section{Heights and height integrals}
\label{sect:heights}

We identify the set $\cA=\{ D_{\alpha}\}$ of boundary components
with $\De=\Delta(\G,\T)$. 
The height pairing 
$$
\Pic(X)_{\C} \times \G(\A)\ra \C
$$
is given by 
$$
H({\bf s}, g) = \prod_v H_v({\bf s}, g_v), 
\quad \text{ where } \quad
H_v({\bf s}, g_v) = \prod_{\alpha \in \mathcal A} |\alpha(t_v)|_v^{s_\alpha}, 
$$
for  $\mathbf s= (s_{\alpha})_{\alpha \in \cA}$, and 
$$
g = (g_v)_v \in \G(\A) \quad \text{ with }  g_v = k_vt_vk_v', \quad 
k_v,k_v' \in \K_v,\quad t_v \in \T(F_v)^+, 
$$
(see Section 6 of \cite{STT}). We will also use
$$
H_S({\bf s}, g) : = \prod_{v \notin S} H_v({\bf s}, g_v).
$$
We have
$$
-K_X = \sum_{\alpha \in \mathcal A} (\kappa_{\alpha}+1)D_{\alpha},
$$
see \cite[Proposition 5.2]{STT}.
For $\epsilon\in \R$, let 
$$
\mathcal{T}_\epsilon =\{ \, {\bf s} \, \mid \,   
\Re(s_\alpha) > \kappa_{\alpha} + 1+\epsilon,\quad  \text{ for all }  \alpha\}
$$
and
$$
\mathcal{T}^\cD_\epsilon = \{ \, {\bf s}  \, \mid \, 
\Re(s_\alpha) > \kappa_{\alpha} + 1+\epsilon, \quad \text{ for all } 
\alpha \not\in \mathcal A_{\cD}\}.
$$

\

The following theorem has been proved in \cite[Sections 6.6 and 6.7]{STT}.

\begin{thm}
\label{ramified} 
Let 
$$
\mathcal{J}_v(\bfs) := \int_{\G(F_v)} H_v(\bfs, g_v)^{-1} \, \mathrm dg_v.
$$
\begin{enumerate} \item For $v\notin
S_{\infty}$ the integral $ \mathcal{J}_v(\bfs)$ is a holomorphic
function for $\bfs \in \cT_{-1}$.
\item For $v\in S_{\infty}$ and $\partial$ in
the universal enveloping algebra the integral
\begin{equation*}
\mathcal{J}_{v,\partial}(\bfs):= \int_{\G(F_v)} \partial( H_v(\bfs,
g_v)^{-1} )\,\mathrm  d g_v
\end{equation*}
is holomorphic for $\bfs \in \cT_{-1}$.
\end{enumerate}
Moreover, for all $\epsilon > 0$ and $\partial$, there exist 
constants $C_v(\eps)$ and  $C_v(\partial, \eps)$ such that 
$$
|\mathcal{J}_v(\bfs)| \leq
C_v(\eps) \quad \text{ and } \quad |\mathcal{J}_{v,\partial}(\bfs)| \leq C_v(\partial, \eps), 
\quad \text{ for all } \quad \bfs \in \cT_{-1 + \eps}.
$$
\end{thm}


The following result generalizes the computations in 
\cite[Sections 6 and 7]{STT} to the $(D,S)$-integral context.

\begin{thm}\label{thm:one-dim}
Let $\chi= \otimes'_v\chi_v$ be a one-dimensional unramified
automorphic representation of $\G(\A)$, 
and $S$ and $\cD$ as in Section~\ref{sect:intro}. 
For each $v\in S$ here exists a function $f_v$, 
holomorphic and uniformly bounded in $\cT_{-1-\delta}$, for some $\delta >0$, 
such that
$$
\int_{\G(F_v)} H_v(\mathbf s, g_v)^{-1} \chi_v \,\mathrm dg_v = \prod_{\alpha\in\cA}
L_v(s_\alpha - \kappa_\alpha,\chi_v \circ \check\alpha) \cdot f_v({\bf s}),
$$
where $L_v$ is the local factor of the Hecke $L$-function. 

Moreover, there exists a function $f_{S,\chi}$, which depends only on 
$(s_{\alpha})_{\alpha \notin \mathcal A_D}$,  is 
holomorphic in $\mathcal T^D_{-1/2}$, uniformly bounded in 
$\mathcal T^D_{-1/2+\epsilon}$, for any $\epsilon >0$, 
and such that
\begin{equation*}
\mathcal{J}_{S, \cD}({\bf s}, \chi) := \!
\int_{\G(\A_S)}  \bdelta_{S}(g)  H_S({\bf s}, g)^{-1}\chi(g) \, \mathrm dg = 
\prod_{\alpha \not\in \mathcal A_\cD}\! L(s_\alpha - \kappa_\alpha,\chi \circ \check\alpha )
\cdot f_{S,\chi}({\bf s}).  
\end{equation*}
\end{thm}

\begin{proof}
The first claim follows from  \cite[Proposition 4.1.2]{C-T-integral}, see also
\cite[Theorem 6.9]{STT}.

To prove the second claim, let
\begin{equation}
\mathcal{J}_{S, \cD}({\bf s}, \chi) = \prod_{v \notin S}
\mathcal{J}_v({\bf s}, \chi_v),
\end{equation}
where
\begin{equation}
\mathcal{J}_v({\bf s}, \chi_v) = \int_{\G(F_v)}\bdelta_v
(g_v) H_v({\bf s}, g_v)^{-1} \chi_v(g_v) \, \mathrm dg_v.
\end{equation}
Since $\G$ is of adjoint type, the collection of elements $\{
\check\alpha(\varpi_v)\}_{\alpha \in \Delta}$ 
forms a basis for the semi-group $\T_v(F_v)^+$.
For ${\bf a} = (a_\alpha)_{\alpha \in \De}$ we set
\begin{equation}
t_v({\bf a}) = \prod_{\alpha \in \mathcal A}\check\alpha(\varpi_v)^{a_\alpha} \in  \T (F_v)^+.
\end{equation}
Using the bi-$\K$-invariance we may write the local integrals as
\begin{equation}
\label{eqn:rewrite}
\mathcal{J}_v({\bf s}, \chi_v) =
\sum_{{\bf a} \in \N^r} \bdelta_v(t_v({\bf a})) 
q_v ^{-\langle{\bf s},  {\bf a}\rangle}  \chi_v(t_v({\bf
a}))\text{vol }(\sK_v t_v({\bf a})\sK_v),
\end{equation}
where 
$$
\langle {\bf s}, {\bf a}\rangle = \sum_{\alpha \in \mathcal A}  s_{\alpha}a_{\alpha}.
$$
By \cite[Lemma 6.11]{STT}, there exists a constant $C$, independent of $v$, 
such that for all
$t_v({\bf a}) \in \T (F_v)^+$, one has
\begin{equation}
\label{eqn:volume}
\vol(\sK_v t_v({\bf a})  \sK_v) \leq \delta_\B(t_v({\bf a}))\left(1 + \frac{C}{q_v}\right),
\end{equation}
where 
$$
\delta_{\B}(t_v({\bf a})):= |\rho(t_v({\bf a}))|_v^2=q_v^{\langle 2\rho, {\bf a} \rangle}.
$$
We may rewrite Equation~\eqref{eqn:rewrite} as
\begin{equation*}
\mathcal{J}_v({\bf s}, \chi_v) =  
\sum_{{\bf a} \in \N^r} \bdelta_v(t_v({\bf a}))
q_v ^{- \langle {\bf s} -2\rho, {\bf a} \rangle} \chi_v(t_v({\bf a}))   + b_v({\bf s}),
\end{equation*}
where
\begin{equation*}
b_v ({\bf s})= \sum_{{\bf a}\in \N^r} \bdelta_v(t_v({\bf a})) 
q_v ^{-\langle {\bf s}, {\bf a}
\rangle}  \chi_v(t_v({\bf
a}))(\text{vol }(\sK_vt_v({\bf a})\sK_v)- \delta_\B(t_v({\bf a}))).
\end{equation*}
Observe that
\begin{equation}\begin{split}\label{delta}
\sum_{\bf a\in \N^r} \!\bdelta_v(t_v({\bf a})) q_v^{-\langle {\bf s -2\rho}, {\bf a}  \rangle} 
\chi_v(t_v({\bf a})) 
& = \prod_{\alpha \not\in \mathcal A_\cD} \biggl( \sum_{a_\alpha
=0}^\infty
\chi_v(\check\alpha(\varpi_v))^{a_\alpha}q_v ^{-(s_\alpha - \kappa_{\alpha}) a_\alpha} \biggr) \\
& = \prod_{\alpha\not\in \mathcal A_\cD} {1 \over 1-
\chi_v(\check\alpha(\varpi_v))q_v^{-(s_\alpha -
\kappa_{\alpha})}}.
\end{split}\end{equation}
The corresponding Euler product, over $v\notin S$, is a product of partial 
$L$-functions, as in the statement of the theorem. 

Let $\mathsf{\sigma}= (\Re(s_\al))_\al$.
In the definition $b_v$, we may assume ${\bf a} \ne
\underline{0}$. Since for each $v \notin S$,
$$
\Biggl\{ {\bf a}\,\, | \,\, {\bf a} \ne \underline{0} \Biggr\} =
\bigcup_{\alpha \in \mathcal A} \Biggl\{ {\bf
a}; a_\alpha \ne 0 \Biggr\},
$$
we have
\begin{align*}
\sum_{v \notin S} \Bigl| b_v({\bf s}) \Bigr| & \leq \sum_{v \notin S}
\sum_\alpha \sum_{a_\alpha \ne 0}\bdelta_v(t_v({\bf a})) 
q_v ^{-\langle {\bf \sigma},  {\bf a}
 \rangle} \Biggl|
(\text{vol }(\sK_vt_v({\bf a})\sK_v)-  \delta_{\B}(t_v({\bf a})))\Biggr|  \\
& \ll\sum_{ v \notin S} q_v^{-1} \sum_{\alpha\not\in \mathcal A_\cD}
\sum_{a_\alpha \ne 0} q_v ^{-\langle  {\bf
\sigma},  {\bf a}  \rangle}
\delta_{\B}(t_v({\bf a})) \\
& =\sum_{v \notin S} q_v^{-1} \sum_{\alpha\not\in \mathcal A_\cD} \biggl(
\sum_{a_\alpha =1}^\infty q_v^{-(\sigma_\alpha -
\kappa_{\alpha})a_\alpha} \biggr) \prod_{ \beta \ne
\alpha \atop \beta \not\in \mathcal A_\cD}
\biggl( \sum_{a_\beta=0}^\infty q_v^{-(\sigma_\beta - \kappa_{\beta})a_\beta} \biggr) \\
& =\sum_{v \notin S} q_v^{-1} \sum_{\alpha\not\in\mathcal A_\cD} {
q_v^{-(\sigma_\alpha - \kappa_{\alpha})} \over
\prod_{\beta\not\in\mathcal A_\cD}
(1 -q_v^{-(\sigma_\beta - \kappa_{\beta})} ) } \\
& \ll \sum_{\alpha\not\in\mathcal A_\cD} \sum_{v \notin S} q_v^{ - {3
\over 2}} < \infty.
\end{align*}

\

Note that for ${\bf s} \in \cT_{-1/2+\epsilon}^\cD$ the estimates are
uniform and the corresponding function $f_{S,\chi}$ is
holomorphic in $\cT^D_{-1/2+\epsilon}$. 
\end{proof}

\section{Infinite-dimensional representations} 
\label{sect:regularintegral}

Let $\pi = \otimes_v' \pi_v$ be an
infinite-dimensional automorphic representation of $\G(\A)$ and
$\varphi_{\pi_v}$ the normalized spherical function associated
to $\pi_v$. Since $\G$ is split, the representation $\pi_v$
is infinite-dimensional, for all $v$ (see  \cite[Proposition 4.4]{STT}). 
We need uniform upper bounds for $\varphi_{\pi_v}$. 
We will use the following special case of a result of Hee Oh 
\cite[Theorem 1.1]{oh}.

\begin{prop}
\label{prop:n} 
Assume that the rank $r$ of $\G$ is at least two. 
Then for each $\alpha \in \De$ and $g_v=k_v t_vk_v'$ we have 
\begin{equation}
\label{eqn:oh}
|\varphi_v(g_v)| \leq |\alpha(t_v)|^{-1/2}
\left( \frac{ (\log_{q_v} \vert \alpha(t_v) \vert ) (q_v-1) +
(q_v+1) } { q_v+1} \right). 
\end{equation}
\end{prop}

\begin{coro}
\label{coro:n} 
Assume that the rank $r$ of $\G$ is at least two. 
Then for each $\alpha \in \De$ we have
\begin{equation*}
\vert \varphi_{\pi_v} (\check\alpha(\varpi_v)) \vert < 2 q_v^{-{1\over 2}}.
\end{equation*}
Moreover, for all $\epsilon >0$ there is a constant
$C_\epsilon >0$ such that for $g_v = k_vt_v k'_v$
we have
$$
|\varphi_{\pi_v}(g_v)| \leq C_\epsilon 
\prod_{\alpha \in \Delta}|\alpha(t_v)|^{-\frac{1}{2r}+\epsilon}.
$$
\end{coro}

\begin{proof}
Equation~\eqref{eqn:oh} implies that for all $\epsilon >0$, 
$$
|\varphi_{\pi_v}(g_v)|\leq C_\epsilon(\alpha) |\alpha(t_v)|^{-\frac{1}{2}+r\epsilon}, 
$$
for some constant $C_\epsilon(\alpha)$.
Multiplying these inequalities over all $\alpha\in \De$ and taking $r$th
root gives the result.
\end{proof}

\

Set 
\begin{equation}
\mathcal{J}_v({\bf s}, \pi_v) =
\int_{\G(F_v)}\bdelta_v(g_v)  H_v({\bf s}, g_v)^{-1}\varphi_{\pi_v}(g_v)\,\mathrm  dg_v.
\end{equation}
Note that for $v\notin S$  this integral 
depends only on $(s_{\alpha})_{\alpha \notin \cA_D}$.

\begin{thm}
\label{non-triv} 
The infinite product
\begin{equation}\label{non-triv-eq}
\mathcal{J}_{S, \cD}({\bf s}, \pi) : = \prod_{v \notin S}
\mathcal{J}_v({\bf s}, \pi_v)
\end{equation}
is holomorphic for ${\bf s}\in \cT^\cD_{-1/2r}$. Moreover, 
for all $\eps>0$ and all compacts $K\subset\cT^D_{-1/2r+\eps}$ 
there exists a constant $C(\eps,K)$, independent of $\pi$, such that
\begin{equation*}
\vert \mathcal{J}_{S, \cD}({\bf s}, \pi) \vert \leq C(\eps,K).
\end{equation*}
for all ${\bf s} \in K$.
\end{thm}

\

\begin{proof}
Using bi-$\K$-invariance, as in the proof of Theorem~\ref{thm:one-dim}, 
we obtain
\begin{equation*}
\mathcal{J}_v({\bf s}, \pi_v) = \sum_{{\bf a}\in \N^r} 
\bdelta_v(t_v({\bf a})) q_v^{-\langle {\bf s}, {\bf a}\rangle }
\varphi_{\pi_v}(t_v({\bf a})) \text{vol}(\K_v t_v({\bf a}) \K_v).
\end{equation*}
Again, $\mathcal{J}_v({\bf s}, \pi_v)$ only depends on $s_{\alpha}$ with $\alpha \notin \cA_D$. 
Using the estimates from Corollary~\ref{coro:n} and Equation~\eqref{eqn:volume}
we conclude that to establish the convergence of the Euler product it suffices to bound
$$
\sum_{v\notin S} \sum_{(a_{\alpha})} 
q_v^{- (\sum_{\alpha \notin \cA_D} (s_\alpha-\kappa_{\alpha} +1/2r-\epsilon)a_{\alpha})},
$$
where the inner sum is over nonzero vectors $(a_{\alpha})\in \N^{\cA\setminus \cA_D}$ .
\end{proof}

Now we consider integrals of the form
\begin{equation*}
\mathcal{J}_v(\bfs, \varphi_{\pi_v}) := \int_{\G(F_v)}
H_v(\bfs, g_v)^{-1} \varphi_{\pi_v}(g_v)\, \mathrm dg_v, 
\end{equation*}
for $v\in S$. 

\begin{thm}
\label{ramified2} \begin{enumerate} \item For all $v\notin
S_{\infty}$ the integral $ \mathcal{J}_v(\bfs, \varphi_v)$ is 
holomorphic for $\bfs \in \cT_{-1-1/2r}$.
\item For $v\in S_{\infty}$ and $\partial$ in
the universal enveloping algebra the integral
\begin{equation*}
\mathcal{J}_{v,\partial}(\bfs, \varphi_{\pi_v}):= \int_{\G(F_v)}
\partial( H_v(\bfs, g_v)^{-1} ) \varphi_{\pi_v}(g)\, \mathrm d g_v
\end{equation*}
is holomorphic for $\bfs \in \cT_{-1-1/2r}$.
\end{enumerate}
\end{thm}
\begin{proof} 
We verify the non-archimedean statement; the other argument is similar. 
Let
$\underline{\sigma}$ be the vector consisting of the real parts of
the components of $\bfs$. Fix $\epsilon >0$.  The local height
integral is majorized by
\begin{align*}
&  \sum_{t \in \sT(F_v)^+} H_v(\underline{\sigma}, t)^{-1} |\varphi_v(t)| \de_B(t) \\
& \ll \prod_{\alpha \in \cA}\sum_{l=0}^\infty H_v(\underline{\sigma}, \check\alpha(\varpi_v^l))^{-1}
q_v^{-(1/2r-\epsilon)l} \delta_B(\check\alpha (\varpi_v^l))\\
& =\prod_{\alpha \in \cA}\sum_{l=0}^\infty q_v^{-(\sigma_\alpha -
\kappa_\alpha+ 1/2r-\epsilon)l }.
\end{align*}
The result is now immediate. 
\end{proof}

\begin{coro}
In the non-archimedean situation, for each $\epsilon > 0$ there is
a constant $C_v(\eps)$, such that $|\mathcal{J}_v(\bfs,
\varphi_{\pi_v})| \leq C_v(\eps)$ for all $\bfs \in \cT_{-1-1/2r+\eps}$.
In the archimedean situation, for all $\eps>0$ and all $\partial$
as above, there is a constant $C_v(\partial, \eps)$ such that
$|\mathcal{J}_{v,\partial}(\bfs, \varphi_{\pi_v})| \leq C_v(\partial,
\eps)$ for all $\bfs \in \cT_{-1-1/2r + \eps}$.
\end{coro}

\begin{coro}\label{coro-inf-int} 
Let $\varphi$ be an automorphic form in
the space of an automorphic representation $\pi$ which is right
invariant under the maximal compact subgroup $\K$. 
Set for $\bfs \in \cT_{\gg 0}$
\begin{equation}\label{inf-int}
\mathcal{J}_{S, \cD}(\bfs, \varphi) : = \int_{\G(\A)}\bdelta_{S,
\cD}(g) H(\bfs, g)^{-1}\varphi(g)dg.
\end{equation}
Then $\mathcal{J}_{S, \cD}(s, \varphi)$ has an analytic continuation
to a function which is holomorphic on $\cT_{-1/2}^\cD \cap
\cT_{-1-1/2r}$. Let $\bDelta$ be the Laplacian as in the proof
of \cite[Lemma 4.1]{Arthur}, and suppose $\varphi$ is an
eigenfunction for $\bDelta$. Define $\bLambda(\phi)$ by
$\bDelta\cdot\varphi = \bLambda(\varphi)\cdot\phi$. Then for each integer
$k>0$, all $\epsilon > 0$, and every compact subset $K \subset
\cT_{-1/2r+\eps}^\cD \cap \cT_{-1-1/2r+\epsilon}$, there exists a
constant $C= C(\eps, K, k)$, independent of $\varphi$, such that
\begin{equation}
|\mathcal{J}_{S, \cD}(\bfs, \varphi)| \leq C \bLambda(\varphi)^{-k}
|\varphi(e)|,
\end{equation}
for all $\bfs \in K$.
\end{coro}

\section{Height zeta function}
\label{sect:heightzeta}

The main tool in the study of
distribution properties of $(S, \cD)$-integral points is the
height zeta function, defined by
$$
\zZ_{S, \cD}({\bf s}, g) := 
\sum_{\gamma \in \G(F)} \bdelta_{S,\cD}(\gamma g) H({\bf s},\gamma g)^{-1}.
$$

\begin{prop}
\label{prop:abs-conv} 
The series defining 
$\zZ_{S, \cD}({\bf s},g)$ converges absolutely to a holomorphic function for ${\bf s}
\in \cT_{\gg 0}$. In its region of convergence
$$
\zZ_{S, \cD}({\bf s},g) \in \mathsf C^\infty (\G(F) \backslash \G(\A)).
$$
and all of its group derivatives are in $\sL^2$. Moreover, in this domain, 
we have a spectral expansion
\begin{equation}
\label{zeta-identity-modified2}
{\sum_{\chi \in \mX,\P}} \frac{1}{n(\sA)}
\int_{\Pi(\M)}\int_{\G(\A)}
\bdelta_{S, \cD}(g')H({\mathbf s},g' )^{-1}\left(\sum_{\phi \in \cB_\P(\pi)_{\chi}} E(g, \phi) 
\overline{E(g',\phi)}\right) \, \mathrm dg'\,
\mathrm d\pi,
\end{equation}
in the notations of \cite[Section 3]{STT}. 
\end{prop}

\begin{proof}
Identical to the proof of \cite[Proposition 8.2]{STT}; 
it suffices to observe that $\zZ_{S,D}$ is a {\em subsum} 
of the series defining the height zeta 
function for rational points considered in \cite{STT}. 
\end{proof}

Let $\mathcal X=\mathcal X(\G)$ be the set of unramified 
automorphic characters of $\G$, i.e., continuous homomorphisms $\G(\A)\ra \mathbb S^1$, 
invariant under $\G(F)$ and $\K$ on both sides. 
We specialize to $g=e$, 
the identity in $\G(\A)$, and obtain
$$
\zZ_{S,D}({\mathbf s}) := \zZ_{S,D}({\mathbf s},e)=\sum_{\chi\in \mathcal X(\G)} 
\int_{\G(\A)} \bdelta_{S, \cD}(g)H({\mathbf s},g )^{-1} \chi(g) \mathrm dg +
S^\flat({\mathbf s}), 
$$
and $S^\flat({\mathbf s})$ is the subsum in Equation~\ref{zeta-identity-modified2}
corresponding to infinite dimensional representations (restricted to $g=e$).

We use this expansion to determine the analytic behavior of the
height zeta function. 
The innermost sum  in the definition of  $S^\flat({\mathbf s})$  is uniformly
convergent for $g'$ in compact sets, see the first half
of the proof of Lemma 4.4 of \cite{Arthur}. Therefore, we may
interchange the innermost summation with the integral over
$\G(\A)$ and find that $S^\flat({\mathbf s})$ equals
\begin{equation*}
{\sum_{\chi \in \mX,\P}}^{\!\!\flat} \frac{1}{n(\sA)}
\int_{\Pi(\M)}\left(\sum_{\phi \in \cB_\P(\pi)_{\chi}} E(e,\phi) 
\int_{\G(\A)}\bdelta_{S,\cD}(g')H({\mathbf s},g' )^{-1} \overline{E(g', \phi)} \, 
\mathrm dg'\right) \, \mathrm d\pi.
\end{equation*}

\begin{thm}
\label{important-theorem}
The function $S^\flat$ admits an analytic continuation to a function
which is holomorphic on $\cT_{-1/2r}^\cD \cap \cT_{-1-1/2r}$, where $r$ is the rank of $\G$. 
\end{thm}

\begin{proof} 
By assumption,  the height function
is invariant under  right and left translation by the maximal compact
subgroup $\K$. By \cite[Corollary 4.1]{STT}, we have
\begin{equation}
S^\flat({\mathbf s}) = {\sum_{\chi \in \mX,\P}}^{\!\!\flat}
\frac{1}{n(\sA)} \int_{\Pi(\M)}\left(\sum_{\phi \in \cB_\P(\pi)_{\chi}}
E(e, \phi) {\mathcal J}_{S, \cD}(\bfs, E(\phi, \cdot))\right) \,
\mathrm d\pi, 
\end{equation} 
where ${\mathcal J}_{S, \cD}(\bfs, E(\phi, \cdot))$ is 
as in Corollary~\ref{coro-inf-int}.
Let $K\subset \cT_{-1/2r+\eps}^\cD \cap \cT_{-1-1/2r + \eps}$, with $\eps>0$,  
be a compact subset.
By Corollary~\ref{coro-inf-int}, for $\bfs \in K$ and all $k$ the
expression
\begin{equation*}
{\sum_{\chi \in \mX,\P}}^{\!\!\flat}  \frac{1}{n(\sA)}
\int_{\Pi(\M)}\left(\sum_{\phi \in \cB_\P(\pi)_{\chi}} |E(e,
\phi)|\cdot| {\mathcal J}_{S, \cD}(\bfs, E(\phi, \cdot))|\right)
\, \mathrm d\pi
\end{equation*}
is bounded by
\begin{equation*}
C(\eps, K, k) {\sum_{\chi  \in \mX,\P}}^{\!\!\flat}  \frac{1}{n(\sA)}
\int_{\Pi(\M)}\left(\sum_{\phi \in \cB_\P(\pi)_{\chi}}
\bLambda(\phi)^{-k} |E(e, \phi)|^2 \right) \, \mathrm d\pi,
\end{equation*}
where $\bLambda(\phi)$ is eigenvalue of the Laplacian $\bDelta$ 
as in Corollary~\ref{coro-inf-int}. 
The convergence of the last expression follows from
\cite[Proposition 3.5]{STT}.
\end{proof}


\begin{coro}
\label{coro:holo}
The height zeta function $\zZ_{S,D}(\mathbf s) = \zZ_{S,D}(\mathbf s, e)$ is holomorphic for 
$\Re(\mathbf s)\in -(K_X+D)  + \Lambda_{\rm eff}(X)^{\circ}$.
\end{coro}

\section{The leading pole}
\label{sect:poles}

We now establish the analog of Manin's conjecture in the context of integral points, proposed
and proved in special cases in \cite{C-T-integral}, \cite{C-T-Ga}, \cite{C-T-Torus}.  
For 
$$
\lambda=\sum_{\alpha\in \mathcal A} \lambda_{\alpha} \alpha \in \Lambda_{\rm eff}(X)^{\circ},
$$ 
the interior of the effective cone of $X$, set
\begin{equation}
a(\lambda) = \max\left(\max_{\alpha\in\mathcal A_{\cD}} \frac{
\kappa_{\alpha}}{\lambda_\alpha}, \max_{\alpha\not\in\mathcal A_{\cD}} \frac{
\kappa_{\alpha}+1}{\lambda_\alpha} \right).
\end{equation}
Let 
$$
\cA(\lambda)=\mathcal A(\lambda, D)
$$ 
be the set of $\alpha$, for which the
maximum is achieved and  $r(\lambda) = \# \cA(\lambda)$ its cardinality. Put
$$
\mathcal A_D(\lambda) = 
\mathcal A(\lambda) \cap \mathcal A_D, \quad d(\lambda) = \# \mathcal A_D(\lambda).
$$
Theorem
\ref{important-theorem} implies that $\zZ_{S,\cD}(s\lambda)$ 
has no pole for $\Re(s) > a(\lambda)$ and that possible contributions
to the right most poles come from one-dimensional
automorphic characters.

Recall that given an automorphic character 
$\chi$ of $\G(\A)$ and an $\alpha \in \Delta(\G, \T)$ 
we can define a Hecke character $\xi_\alpha(\chi)$ of $\mathbb G_m(\A)$ by 
$$
\xi_\alpha(\chi) = \chi \circ \check\alpha.  
$$
Then if $\chi = \otimes_v'\chi_v$, we have $\xi_\alpha(\chi) = \otimes_v'\xi_{\alpha, v}(\chi_v)$ with 
$$
\xi_{\alpha, v}(\chi_v) = \chi_v \circ \check\alpha.
$$

\

We are only interested in those automorphic $\chi$ which satisfy
\begin{equation}
\int_{\G(\A)}\bdelta_{S, \cD}(g) H(s\lambda, g)^{-1} \chi(g) \, \mathrm dg \ne 0
\end{equation}
for some $s$ in the domain of absolute convergence. This implies
that $\chi$ is right, and in this case also left, invariant under
the maximal compact subgroup $\sK$ of $\G(\A)$, i.e., $\chi\in \mathcal X(\G)$.

\begin{defn}
\label{defn:chi}
Let $\mathcal X_{S, \cD, \lambda}(\G)\subset \mathcal X(\G)$ 
be the collection of all characters $\chi=\otimes'_v\chi_v$ 
such that 
\begin{itemize}
\item 
for all $\alpha \in \cA(\lambda) \setminus \cA_\cD(\lambda)$ and all $v\notin S$, 
we have $\xi_{\alpha,v}(\chi_v) \equiv 1$;
\item 
for all $\alpha \in \cA_\cD(\lambda)$, and $v \in S$, we have $\xi_{\alpha, v}(\chi_v) \equiv 1$. 
\end{itemize}
\end{defn}

\

\begin{rem}
\label{rema:trivial}
If $D=\emptyset$ and $\lambda=-K_X$ then 
$\mathcal X_{S,D, \lambda}(\G)=1$, the trivial character, for any $S$ \cite[Proposition 8.6]{STT}.
If $D=\cup_{\alpha \in \cA}$ is the whole boundary and $S=S_{\infty}$ then 
$\mathcal X_{S,D,\lambda}(\G)$ is dual to the {\em class group} of $\G$, i.e., the quotient 
$$
{\rm cl}(\G):=\G(\A)/\G(F)\K\prod_{v\mid \infty} \G(F_v).
$$
This class group is trivial for groups of type $E_8$, $F_4$, and $G_2$, as the adjoint groups
are also simply-connected, and these have class number 1.   
\end{rem}

\begin{exam}
\label{weeteck}
If $\chi$ is an unramified character of a split semi-simple group
over a number field of class number one, then $\chi =1$;
this follows from  \cite[Lemma 4.7]{GMO}, 
and Corollary 2 on page 486 of \cite{pr}.

This may fail when the class number of $F$ is not 
equal to one.  E.g., let $F$ be a field with class number
two, and let $E$ be the Hilbert class field of $F$. Then $E/F$ is
an unramified quadratic extension. Let $\omega_{E/F}$ be the
corresponding quadratic character of $\A_F^\times$. Consider the
automorphic character of $\PGL(2)$ given by $\omega_{E/F} \circ
\det$. Then this automorphic character is unramified and 
trivial at the archimedean places. Such characters will contribute to the leading pole
of $\zZ_{S,D}$. 
\end{exam}

\

We have shown that
\begin{equation}
\zZ_{S, \cD}(s\lambda)=
\sum_{\chi\in \mathcal X(\G)}
\int_{\G(\A)}\bdelta_{S, \cD}(g)H(s\lambda, g)^{-1}\chi(g)\, \mathrm dg + f(s),
\end{equation}
with $f$ holomorphic for $\Re(s) >a(\lambda) - \delta$, for some $\delta >0$.
Theorem~\ref{thm:one-dim} combined with basic properties of Hecke
$L$-functions shows that for $\chi \in \mathcal X(\G)$ the integral
$$
\int_{\G(\A)} \bdelta_{S, \cD}(g) H_S(s\lambda, g)^{-1} \chi(g) \, \mathrm d g
$$
admits a regularization of the shape
$$ 
\prod_{\alpha \in \mathcal A(\lambda) \setminus \cA_\cD(\lambda)}  \!\!\!\!
L_S(s\lambda_\alpha - \kappa_\alpha,\xi_{\alpha}(\chi))\cdot h_{\chi}(s) \cdot 
\prod_{v\in S} \prod_{\alpha\in \cA_D(\lambda)}\!\!\!
L_v(s\lambda_\alpha - \kappa_\alpha, \xi_{\alpha,v}(\chi_v))\cdot h_{\chi,v}(s), 
$$
with $h_{\chi}$ and $h_{\chi,v}$ holomorphic for $\Re(s) > a(\lambda)-\delta$, for
some $\delta >0$. 
It follows that only $\chi\in\mathcal X_{S, \cD, \lambda}(\G)$ 
contribute to the leading term at $s=a(\lambda)$. 
By Poisson summation formula, we can rewrite this contribution as 
$$
|\mathcal X_{S,D,\lambda}(\G)|\cdot 
\int_{\G(\A)^{{\rm Ker}_\lambda}} \bdelta_{S, \cD}(g)H(s\lambda, g)^{-1}\, \mathrm dg,
$$
where 
$$
\G(\A)^{{\rm Ker}_\lambda}:= \cap_{\chi\in \mathcal X_{S,D,\lambda}(\G)} \Ker(\chi)
$$
is the intersection of the kernels of automorphic characters.

\begin{thm}
\label{thm:main2}
The number of $(S, D)$-integral point of bounded height with respect to $\lambda$ is
asymptotic to
$$
c\cdot B^{a(\lambda)}\log(B)^{b(\lambda)-1}(1+o(1)), \quad \quad B\ra \infty,  
$$
where 
$$
b(\lambda)=  r(\lambda)-d(\lambda) + \sum_{v\in S} d(\lambda), 
$$
and 
\begin{equation}
\label{eqn:ccc}
c =\frac{1}{a(\lambda)(b(\lambda)-1)!}\cdot |\mathcal X_{S,D,\lambda}(\G)| \cdot
\int_{\G(\A)^{{\rm Ker}_\lambda}} \bdelta_{S, \cD}(g)H(s\lambda, g)^{-1}\, \mathrm dg >0.
\end{equation}
\end{thm} 

\begin{proof}
We adopt the proof in \cite[Theorem 9.2]{STT}, combining it with
a Tauberian theorem as in \cite[Theorem A.15]{C-T-integral}. 
It suffices to establish that the limit
$$
\lim_{s\ra a(\lambda)} (s- a(\lambda))^{b(\lambda)} 
\int_{\G(\A)^{{\rm Ker}_\lambda}}\bdelta_{S, \cD}(g) H(s\lambda, g)^{-1}\, \mathrm dg  >0.
$$
Note that there exists a finite set of $\gamma_j\in \G(F)$ 
such that 
$$
\G(\A) = \sqcup \,\, \gamma_j \G(\A)^{{\rm Ker}_\lambda}
$$
and that the local and global integrals over each of these cosets are comparable, 
upto a constant. In particular, it suffices to establish that 
$$
\lim_{s\ra a(\lambda)} (s- a(\lambda))^{b(\lambda)} 
\int_{\G(\A)}\bdelta_{S, \cD}(g) H(s\lambda, g)^{-1}\, \mathrm dg  >0,
$$
which follows from Theorem~\ref{thm:one-dim} and the definitions.
\end{proof}

We now specialize to the case when 
$\lambda=-(K_X+D)$, the log-anticanonical line bundle.
The first condition in Definition~\ref{defn:chi} 
implies that if $\chi \in \mathcal X_{S,D,\lambda}(\G)$ and
$\alpha \notin \cA_D$ then
$\xi_{\alpha,v}(\chi_v)\equiv 1$ for {\em all} $v$. 
Combining with the second condition, 
we see that 
$\xi_{\alpha, v}(\chi_v) \equiv 1$ for all $\alpha\in \cA$ and all $v\in S$. 
In particular, for $\chi\in \mathcal X_{S,D,\lambda}(\G)$ the integrals
$$
\int_{\G(\A)}  \bdelta_{S, \cD}(g)H(s\lambda, g)^{-1}\chi(g)\, \mathrm dg
$$
do not depend on $\chi$ and equal
$$
\int_{\G(\A_S)}  \bdelta_{S, \cD}(g)H(s\lambda, g)^{-1}\, \mathrm dg \cdot 
\prod_{v\in S}\int_{\G(F_v)} H_v(s\lambda, g_v)^{-1}\, \mathrm dg_v
$$
We can now describe the constant $c$ appearing in \eqref{eqn:ccc} 
in terms of Tamagawa-type constants. 
First,  we recall some notation. 
By \cite{concini-p83}, the boundary strata of $X\setminus \G$ are in 
bijection with subsets
$A\subset \cA$, i.e., there is a unique stratum 
$$
Z_{A}:=\cap_{\alpha\in A} D_{\alpha}.
$$ 
For split $\G$,
each such stratum $Z_{A}$ contains $F_v$-adic points. 

In the terminology of \cite[Section 3]{C-T-integral}, at each place $v$, 
the {\em analytic Clemens polytope}  $\mathcal C^{\rm an}_v(D)$ of
$D$ has a unique face of maximal dimension, it corresponds to $Z_{\cA_D}(F_v)$. 
Put $d:=\dim(\mathcal C^{\rm an}_v(D))+1$; we have $d=\#\cA_D$, the
codimension of the stratum  $Z_{\cA_D}$  
(see also \cite[Section 5.3.2]{C-T-integral}). 
In this situation, for each $v\in S$, 
there is a distinguished $v$-adic measure $\tau^{\rm max}_v$ on $D(F_v)$
considered in \cite[Section 4]{C-T-integral}. It is
supported on $Z_{\cA_D}(F_v)$ and the corresponding 
volumes are given by
$$
\tau^{\rm max}_v(D(F_v)) =\prod_{\alpha\in \cA_D} \frac{1}{\kappa_{\alpha}}  
\cdot \lim_{s\ra 1} (s-1)^d \int_{\G(F_v)} H_v(s\lambda,g_v)^{-1}\, \mathrm dg_v.
$$ 
Furthermore, there is an {\em adelic} measure on the integral adeles on $U:=X\setminus D$, 
which in our case takes the form:
$$
\tau^S_{(X,D)}(U(\A_S)^{\rm int}) = \prod_{\alpha\notin \cA_D} 
\frac{1}{\kappa_{\alpha}+1}\cdot 
\lim_{s\ra 1} (s-1)^{r-d} \int_{\G(\A_S)}\bdelta_{S}(g) H_S(s\lambda, g_v)^{-1} \, \mathrm dg. 
$$
It follows that 
\begin{equation}
\label{eqn:c}
c= \frac{1}{(b-1)!} \cdot |\mathcal X_{S,D,\lambda}(\G)|   \cdot \tau^S_{(X,D)}(U(\A_S)^{\rm int}) \cdot \prod_{v\in S} \tau^{\rm max}_v(D(F_v)),
\end{equation}
where 
$$
b:=(r-d) + \sum_{v\in S} d.
$$
Formula \eqref{eqn:c} interpolates between Peyre's 
Tamagawa-type constant for leading terms in asymptotics of rational points and 
the ``concentration of counting measures to the Satake boundary'' for asymptotics
for integral points on $\G$, established in \cite{GOS}.

\bibliographystyle{amsplain}
\bibliography{myfile2}

\end{document}